\documentclass[12pt]{amsart}

\usepackage{amsfonts,amsthm,amsmath,amssymb,amscd,mathrsfs,tikz}
\usepackage{graphics}
\usepackage{indentfirst}
\usepackage{cite}
\usepackage[title]{appendix}
\usepackage{latexsym}
\usepackage[dvips]{epsfig}
\usepackage{color}
\usepackage{enumerate}
\newtheorem{theorem}{Theorem}[section]
\newtheorem{remark}{Remark}[section]

\newtheorem{lemma}[theorem]{Lemma}

\newcommand{\bt}{\begin{theorem}}
	\newcommand{\bl}{\begin{lemma}}
		\newcommand{\el}{\end{lemma}}
	\newcommand{\et}{\end{theorem}}

\newcommand{\bn}{\begin{eqnarray}}
	\newcommand{\en}{\end{eqnarray}}
\newcommand{\bnn}{\begin{eqnarray*}}
	\newcommand{\enn}{\end{eqnarray*}}

\newcommand{\ba}{\begin{aligned}}
	\newcommand{\ea}{\end{aligned}}
\newcommand{\be}{\begin{equation}}
	\newcommand{\ee}{\end{equation}}

\newcommand{\Bv}{{\boldsymbol{v}}}
\newcommand{\bBV}{\boldsymbol{V}}

\newcommand{\Bh}{{\boldsymbol{h}}}
\newcommand{\Bu}{{\boldsymbol{u}}}
\newcommand{\Bw}{{\boldsymbol{w}}}
\newcommand{\Be}{{\boldsymbol{e}}}

\newcommand{\Ps}{\mathbf{\Psi}}

\newcommand{\Bx}{{\boldsymbol{x}}}

\newcommand{\OR}{{\mathscr{O}}_{R}}

\newcommand{\DR}{{\mathcal{D}}_{R}}

\begin{document}
	
	\title[Liouville-type theorem]
	{Liouville-type theorem for   steady helically symmetric   MHD   system in $\mathbb{R}^{3}$}

	\author{Jingwen Han}
	\address{School of Mathematical Sciences, Shanghai Jiao Tong University, 800 Dongchuan Road, Shanghai, China}
	\email{hjw126666@sjtu.edu.cn}
%
%
%
%
%
%
%

	\begin{abstract}
	 We show that any bounded smooth helically symmetric solution $(\Bu, \Bh)$ in $\mathbb{R}^3$  must be  constant vectors. This is an extension of  previous result \cite[Theorem 1.1]{HWXAHE} from Navier-Stokes system to MHD system. The  proof relies on establishing  a Saint-Venant  type estimate to characterize the growth of Dirichlet integral of nontrivial solutions.
	\end{abstract}

	\keywords{Liouville-type theorem, steady MHD system, helical symmetry,   Saint-Venant  type estimate.}
	\subjclass[2010]{
	35B53,	  35B10,  35J67, 76W05,  76D03}

	
	\maketitle
	
	\section{Introduction and  Main Results}
	 In this paper, we are interested in the Liouville-type theorem for solutions to the  steady  magnetohydrodynamic (MHD) system in $\mathbb{R}^{3}$,
	\begin{equation}\label{eqsteadyMHD}
	\left\{ \ba
	& -\Delta \Bu + (\Bu \cdot \nabla )\Bu-(\Bh\cdot\nabla)\Bh + \nabla P  =0,  \\
	& -\Delta \Bh + (\Bu \cdot \nabla )\Bh-(\Bh\cdot\nabla)\Bu= 0,              \ \ \ \ \  &\quad \mbox{in}\ \mathbb{R}^{3},\\
	& \nabla \cdot \Bu=\nabla \cdot \Bh =0,    \\
	\ea \right.
	\end{equation}
 where the unknown function $\Bu=(u^1,u^2,u^3)$  is the velocity field, $\Bh=(h^1,h^2,h^3)$  is the magnetic field, $P$ is the pressure.   The MHD system describes the motion of electrically conducting fluids, such as plasmas. It is an indispensable nonlinear partial differential equations in nature.  One may refer to \cite{DDSMHDBOOK09}  for more background related to MHD system.
	
 As we all know, if $\Bh=0$, the MHD system \eqref{eqsteadyMHD} reduces to the steady Navier-Stokes system. The weak solution $\Bu$ is called D-solution to the Navier-Stokes system if the finite Dirichlet integral 
 \[
 \int_{\Omega}|\nabla\Bu|^{2}\,d\Bx<+\infty
 \]
holds, where $\Omega$ is $\mathbb{R}^{3}$ or some noncompact domains. Leray  \cite{LEJMPA33} proved the existence of D-solution in both bounded and exterior domains. A Liouville-type  problem \cite[Remark X. 9.4]{GAGP11} asks  whether the  D-solution is zero if the velocity at infinity is zero in  $\mathbb{R}^{3}$. There are some progress along this problem in whole space for example the two-dimensional  case and  axisymmetric three-dimensional without swirl  case were solved in \cite{GWASP78} and \cite{KPRJMFM15}, respectively. In the noncompact domain for example the slab $\mathbb{R}^{2}\times[0,1]$, it was proved in the slab domain supplemented  with the no-slip boundary conditions the D-solutions must be zero (cf.\cite{CPZZARMA20}). The authors also  obtained the Liouville-type theorem of axisymmetric D-solutions  in the periodic slab $\mathbb{R}^{2}\times \mathbb{T}$. 

Another important problem is classifying the $L^{\infty}$-bounded solution of the Navier-Stokes equations. The significance progress in \cite{KNSS09} indicated that any bounded two dimensional and three dimensional axisymmetric without swirl solutions in whole space must be constant.
In the $z$-periodic slab   $\mathbb{R}^{2}\times \mathbb{T}$, it is shown that the bounded axisymmetric solution must be trivial provided  $ru^{\theta}$ is bounded. In the slab  $\mathbb{R}^{2}\times[0,1]$, the Liouville-type theorems for bounded axisymmetric steady solutions have been established in \cite{aBGWX,HWXAHE}. In particular, it was proved in \cite{aBGWX} that the      axisymmetric   bounded solutions in  $\mathbb{R}^{2}\times \mathbb{T}$ must be trivial.

 Analogously to the Navier-Stokes system, 
we  consider the weak solution to the MHD system \eqref{eqsteadyMHD} with finite Dirichlet integral
\begin{equation}\label{MHDfinDinte}
	\int_{\Omega}|\nabla\Bu|^{2}+|\nabla\Bh|^{2}\,d\Bx<+\infty.
\end{equation}
For the MHD system in the whole spaces, it was proved  that the D-solutions are trivial under  additional assumptions (cf. \cite{DHS22Nonli,WWYWNONL19,CWDCDS16,LLNoDEA2021}).  As for the periodic slab  $\mathbb{R}^{2}\times \mathbb{T}$ domain,
Li and Pan proved the triviality of D-solutions provided ($\Bu$, $\Bh$) are  swirl-free or axisymmetric  in \cite{LPJMFM2019}. Later on, \cite{PJMP21,HWJZJMAA23} improved this result by assuming the axisymmetry of component. 
In \cite{HWJZJMAA23}, they also considered the bounded solutions in the periodic slab $\mathbb{R}^{2}\times \mathbb{T}$, which is an extension of 
\cite[Theorem 1.4]{aBGWX} from the Navier-Stokes system to the MHD system.  As far as we know,  it is difficult to generalize the 2D and 3D axisymmetry  without swirl bounded solutions cases  from the Navier-Stokes system  to the MHD system without other assumptions, due to the lack of the maximum principle. See, for example,  \cite{WJMFM21}.  However,  there are still lots of results concerning the Liouville-type theorems to steady MHD system, one may refer \cite{CWJDE21,CLWJNS22,CNYNonl2024,FWJMP2021,LLNBKMC2020,WZCPAA2023,WGJGA2024,SsAMSS19,YX20JMAA} and references therein.
 
  The focus of this paper is to obtain the Liouville-type theorem for helically symmetric  MHD flows. 
  To give the definition of the helical symmetry, we first denote the 
    rotation matrix $R_{\rho}$ by an angle $\rho$ around the $x_{3}$-axis, i.e.
  \[
  R_{\rho}=\left(
  \begin{array}{ccc}
  	\cos\rho  &      \sin\rho & 0 \\
  	-\sin\rho & 	\cos\rho & 0 \\
  	0         &            0 & 1
  \end{array}
  \right).
  \]
The helical symmetry group $\mathcal{G}_{\sigma}$ is a one-parameter group of isometries of $\mathbb{R}^{3}$,
	\begin{equation}\label{eqA2}
		\mathcal{G}_{\sigma}=
		\left\{
		S_{\rho}:\mathbb{R}^{3}
		\rightarrow
		\mathbb{R}^{3} \big|  \rho\in\mathbb{R}
		\right\},
	\end{equation}
	here the transformation $S_{\rho}$  is defined by
	\begin{equation}\label{eqA3}
		S_{\rho}(\Bx)=
		R_{\rho}(\Bx)
	     +\left(
		\begin{array}{c}
			0 \\
			0 \\
			\sigma \rho
		\end{array}
		\right),
	\end{equation}
 and the nonzero constant $\sigma$ is called the step or pitch, which denotes the translation along the $x_{3}$-axis. 
Helical symmetry is invariant under a one-dimensional helical symmetry group $\mathcal{G}_{\sigma}$ of rigid motions generated by a simultaneous rotation around a fixed axis and translation along the same  symmetric axis.
A smooth function $f(\Bx)$ and vector field $\Bu(\Bx)$ are helically symmetric if $f\left(S_{\rho}(\Bx)\right)=f(\Bx)$ and $\Bu(S_{\rho}(\Bx))= R_{\rho}\Bu(\Bx)$,  for all $\rho\in \mathbb{R}$, respectively. Hence any helically symmetric flow is  periodic in $x_{3}$-direction, with   a period $2\pi|\sigma|$. Without loss of generality, we will assume that $\sigma=\frac{1}{2\pi}$.
	
There are  many interesting works related to the helical fluid, see \cite{BLLN15,ETSIAM09,JLNJDE17,JMCLN18PHYD,LNZAMP17,LMNLTJDDE14}.  For the non-stationary Navier-Stokes system, the global existence and uniqueness of  strong helical solutions with helical initial data were initially established in \cite{MTLARMA90}. In \cite{BLNNT13}, the authors obtained the stability of the helical solutions. To the stationary Navier-Stokes system, the existence of helical solutions with helical external force and boundary value were shown in \cite{KLWSIAM22}.  
 The Liouville-type theorem shows  that the bounded smooth helical solutions to the steady Navier-Stokes system in $\mathbb{R}^{3}$ must be constant \cite{HWXAHE}. It should be mentioned that in \cite{GGNAMAS14} the existence, uniqueness and large time behavior of  three-dimensional helically symmetric MHD system are investigated.

 Our  main result is the Liouville-type theorem for bounded helically symmetric MHD flows.
	\begin{theorem}\label{th:01}
		Assume that  $(\Bu, \Bh)$ is a   bounded smooth helically symmetric solution to the MHD   system \eqref{eqsteadyMHD} in $\mathbb{R}^{3}$. Then $(\Bu, \Bh)$ must be  constant vectors of the form  $\Bu=\left(0, 0, C_{1}\right)$, $\Bh=\left(0, 0, C_{2}\right)$.
	\end{theorem}

Indeed, Theorem \ref{th:01} can be improved as follows.
\begin{remark}\label{remarkreq1}
	Assume that  $(\Bu, \Bh)$ is a   bounded smooth  solution to the MHD   system \eqref{eqsteadyMHD} in $\mathbb{R}^{3}$, if $\Bu$ is helically symmetric and $\Bh$ is periodic in $z$-variable, then $(\Bu, \Bh)$ must be  constant vectors. Furthermore, the constant vector $\Bu$ must be the form of $\left(0, 0, C\right)$.
\end{remark}

In terms of the  cylindrical coordinates $(r,\theta,z)$, which are defined as follows
\begin{equation}\label{eq6}
	\Bx=(x_{1}, x_{2}, x_{3})
	=(r\cos\theta, r\sin\theta,z),
\end{equation}
 the  component $\Bu$ and  $\Bh$ can be written as
\begin{equation}\label{eq7}
	\Bu=u^{r}(r,\theta,z)\Be_{r}
	   +u^{\theta}(r,\theta,z)\Be_{\theta}
	   +u^{z}(r,\theta,z)\Be_{z},\quad
		\Bh=h^{r}(r,\theta,z)\Be_{r}
	+h^{\theta}(r,\theta,z)\Be_{\theta}
	+h^{z}(r,\theta,z)\Be_{z},   
\end{equation}
where $u^{r}, u^{\theta}, u^{z}$ are called radial, swirl and axial velocity, respectively, with
\[
\Be_{r}=(\cos\theta, \sin\theta, 0), \ \ \ \Be_{\theta}=(-\sin\theta, \cos\theta, 0) \ \ \ \text{and} \ \ \
\Be_{z}=(0, 0, 1).
\]	
We say that $\Bu$ is axisymmetric, if $\Bu$ does not depedend on $\theta$, i.e. 
\[
	\Bu=u^{r}(r,z)\Be_{r}
+u^{\theta}(r,z)\Be_{\theta}
+u^{z}(r,z)\Be_{z}.
\]
With the aid of cylindrical coordinates, the momentum equation $\eqref{eqsteadyMHD}_{1}$  can be written as
{\scriptsize{\begin{equation}\label{eqA8}
	\left\{
\begin{aligned}
	&\left(u^r \partial_r + \frac{u^\theta}{r}  \partial_\theta + u^z \partial_z\right)u^r - \frac{(u^\theta)^2}{r} + \frac{2}{r^2} \partial_\theta u^\theta + \partial_r P =\left(h^r \partial_r + \frac{h^\theta}{r}  \partial_\theta + h^z \partial_z\right)h^r - \frac{(h^\theta)^2}{r}+ \left(\Delta_{r,\theta, z} - \frac{1}{r^2} \right) u^r,\\
	&\left(u^r \partial_r + \frac{u^\theta }{r}  \partial_\theta + u^z \partial_z\right)u^\theta +
	\frac{u^\theta u^r}{r} - \frac{2}{r^2} \partial_\theta u^r + \frac{1}{r} \partial_\theta P=\left(h^r \partial_r + \frac{h^\theta}{r}  \partial_\theta + h^z \partial_z\right)h^\theta +
	\frac{h^\theta h^r}{r}+  \left(\Delta_{r,\theta, z} - \frac{1}{r^2} \right) u^\theta,\\
	&\left(u^r \partial_r + \frac{u^\theta}{r} \partial_\theta + u^z \partial_z\right)u^z + \partial_z P =\left(h^r \partial_r + \frac{h^\theta}{r} \partial_\theta + h^z \partial_z\right)h^z +\Delta_{r,\theta, z}u^z,
	\end{aligned}
\right.
\end{equation}}}
where
\[
\Delta_{r,\theta,z}
=\partial^{2}_{r}+\dfrac{1}{r}\partial_{r}
+\dfrac{1}{r^{2}}\partial^{2}_{\theta}
+\partial^{2}_{z}.
\]
The divergence free condition $\eqref{eqsteadyMHD}_{3}$ is 
\begin{equation}\label{eqdivecylinc}
\partial_r u^r +\frac{1}{r}\partial_\theta u^\theta+\partial_z u^z +\frac{u^r}{r}=\partial_r h^r +\frac{1}{r}\partial_\theta h^\theta+\partial_z h^z +\frac{h^r}{r}=0. 
\end{equation}

Now we outline the key point of the proof. The strategy of proving Theorem \ref{th:01} is obtaining the Saint-Venant type estimates for the Dirichlet integral of nontrivial solutions, which making use of the helical identities \eqref{dhelicdef} and periodicity  of helically MHD flows and the Bogovskii map.	The Saint-Venant's principle was firstly used  in \cite{JKKTARMA66,RATARMA65} to study the solutions of the elasticity system. Then the authors in \cite{HWSIAM78,LSZNSL80}  utilized the Saint-Venant's principle to study the decay estimate and the the Liouville-type theorems for the pipe  flows. Recently,  it was successfully applied in the slab domain $\mathbb{R}^{2}\times [0,1]$ to study the Liouville-type theorems for  bounded axisymmetric Navier-Stokes solutions \cite{aBGWX}. There are many extensions in studying the Liouville-type theorems by  Saint-Venant type estimates, see \cite{aBYA,KTWarXi23,HWXAHE,HWJZJMAA23}.

The organization of the  paper is as follows.  In Section \ref{Sec2}, we will state some preliminaries. The proof of the main result will be given in Section \ref{Sec3}. In the Appendix, we will give a short note of Remark \ref{remarkreq1}.

	\section{Preliminaries}\label{Sec2}
	In this section, we give some preliminaries. First, we introduce the following notations. Assume that $\Omega$ is a bounded domain, define
	\[
	L^p_0(\Omega)=\left\{g: \ \ g\in L^p(\Omega), \ \ \int_\Omega g \, d\Bx =0 \right\}.
	\]
	For any $R\geq 2$, denote $D_R =(R-1, R)\times(0, 1)$,
 $\mathcal{D}_R =(R-1, R)\times(0, 2\pi)\times(0, 1)$,  $\Omega_R = B_R \times (0, 1)$ and $\OR = (B_R \setminus \overline{B_{R-1}}) \times (0, 1)$, where $B_R= \{(x_1,x_2)\in \mathbb{R}^2: x_1^2+x_2^2 <R^2\}$.

 Then, we introduce the Bogovskii map, which gives a solution to the divergence equation. The proof is due to  Bogovskii \cite{B79DANS}, see also  \cite[Section III.3]{GAGP11}  and  \cite[Section 2.8]{TT18}.
	 \begin{lemma}(\hspace{1sp}\cite[Lemma 2.1]{aBGWX})\label{Bogovskii}
		Let $\Omega$ be a bounded Lipschitz domain in $\mathbb{R}^n$ with $n\geq 2 $.  For any $q\in (1, \infty)$, there is a linear map $\boldsymbol{\Phi}$ that maps a scalar function $g\in L^q_0(\Omega)$ to a vector field $\bBV = \boldsymbol{\Phi} g \in W_0^{1, q}(\Omega; \mathbb{R}^n)$ satisfying
		\be \nonumber
		{\rm div}~\bBV = g \ \text{in}\ \Omega \quad \text{and} \quad \|\bBV\|_{W^{1, q}(\Omega)} \leq C (\Omega, q) \|g\|_{L^q(\Omega)}.
		\ee
In particular,
		\begin{enumerate}
			\item    For any $g \in L^2_0(D_R)$,
			the  vector valued function $\bBV = \boldsymbol{\Phi} g \in H_0^1(D_R; \mathbb{R}^2)$ satisfies
			\be \nonumber
			\partial_r V^r + \partial_z V^z =g \ \  \mbox{in}\,\, D_R
			\quad
			\text{and}
			\quad
			\|\tilde{\nabla } \bBV\|_{L^2(D_R)}
			\leq C \|g\|_{L^2(D_R)},
			\ee
			where $\tilde{\nabla } = (\partial_r, \ \partial_z ) $ and $C$ is a constant independent of $R$.
			
			\item  For any $g \in L^2(\mathcal{D}_R)$,
			the   vector valued function $\bBV = \boldsymbol{\Phi} g \in H_0^1( \mathcal{D}_R; \mathbb{R}^3)$ satisfies
			\be \nonumber
			\partial_r V^r + \partial_\theta V^\theta +  \partial_z V^z =g \ \   \mbox{in}\,\, \mathcal{D}_R
			\quad
			\text{and}
			\quad
			\|\bar{\nabla } \bBV\|_{L^2(\mathcal{D}_R)}
			\leq C \|g\|_{L^2( \mathcal{D}_R)},
			\ee
			where $\bar{\nabla} = (\partial_r, \ \partial_\theta, \ \partial_z ) $ and $C$ is a constant independent of $R$.
		\end{enumerate}
	\end{lemma}
	
  Next, we give some important helical identities, which will be used frequently in the proof.
	\begin{lemma}\label{Leheliclfl}
		If the continuously differentiable velocity field $\Bu=u^{r}(r, \theta, z) \Be_{r}+
		u^{\theta}(r,\theta,z)\Be_{\theta}+
		u^{z}(r,\theta,z)\Be_{z}$ and the magnetic field $\Bh=h^{r}(r, \theta, z) \Be_{r}+
		h^{\theta}(r,\theta,z)\Be_{\theta}+
		h^{z}(r,\theta,z)\Be_{z}$ are helically symmetric, then there exists a constant $\sigma \in \mathbb{R}$  such that
		\begin{equation}\label{dhelicdef}
		    \sigma\partial_{z}\Bu
				=\partial_{\theta}\Bu,
		\qquad	\sigma\partial_{z}\Bh=\partial_{\theta}\Bh.
		\end{equation}
	\end{lemma}
	\begin{proof}
  According to \cite{GGNAMAS14,HWXAHE}, the continuous vector fields $(\Bu, \Bh)$ are helically symmetric if and only if there exist $\left(\Bv, \Bw\right)$ and  a constant $\sigma \in \mathbb{R}$ such that
		\begin{equation}\label{helidef}
		\Bu(r,\theta,z)= 	\Bv(r,\sigma\theta+z),\qquad
		\Bh(r,\theta,z)= 	\Bw(r,\sigma\theta+z).
		\end{equation}
	Differentiating \eqref{helidef} with respect to $\theta$ and $z$ variables  gives \eqref{dhelicdef}.  The proof of  Lemma \ref{Leheliclfl} is completed.
	\end{proof}
	For a solution of MHD system \eqref{eqsteadyMHD} in a slab with periodic boundary conditions, if the velocity field and magnetic filed is $L^{\infty}$-bounded, then its gradient  must also be $L^{\infty}$-bounded.
	Here we omit the proof, which can be found in \cite[Lemma 2.4]{HWJZJMAA23}. 
	\begin{lemma}\label{le:pebu}
		Let $(\Bu, \Bh)$ be  bounded smooth solutions to  the  MHD  system \eqref{eqsteadyMHD} in $\mathbb{R}^{2}\times \mathbb{T}$. Then $(\nabla\Bu, \nabla\Bh)$ are  bounded.
	\end{lemma}

	\section{Helically Symmetric MHD Flows}\label{Sec3}
 Since the helically symmetric MHD flow is periodic along the axial direction,  some observations in  \cite{HWJZJMAA23} are employed in this section. Before proving Theorem \ref{th:01},  we first insert
a lemma, which shows that the D-solution of  helically symmetric MHD system must be  constant.
	\begin{lemma}\label{Le:32}
		Let $(\Bu, \Bh)$ be a bounded smooth helically symmetric solution to the  MHD system \eqref{eqsteadyMHD} in $\mathbb{R}^{3}$, 
		with a finite Dirichlet integral in the slab, i.e.,
		\begin{equation}\label{21Dintass}
			\int_{\mathbb{R}^{2}\times
				(0,\, 1)}(|\nabla\Bu|^{2} +|\nabla\Bh|^{2})\,d\Bx<+\infty.
		\end{equation}
	Then $(\Bu, \Bh)$ must be  constant vectors of the form $\Bu=\left(0, 0, C_{1}\right)$, $\Bh=\left(0, 0, C_{2}\right)$.
	\end{lemma}
\begin{proof}[Proof]
	The proof is divided into two steps.
	
	\emph{Step 1.} \emph {Set up.}
	For  $R\geq 2$ sufficiently large, we define the  cut-off function $\varphi_R(r)$,
	\be \label{cut-off}
	\varphi_R(r) = \left\{ \ba
	&1,\ \ \ \ \ \ \ \ \ \ r < R-1, \\
	&R-r,\ \ \ \ R-1 \leq r \leq R, \\
	&0, \ \ \ \ \ \ \ \ \ \ r > R.
	\ea  \right.
	\ee
 Multiplying  $\eqref{eqsteadyMHD}_{1}$ and $\eqref{eqsteadyMHD}_{2}$  by $\varphi_{R}(r)\Bu$ and $\varphi_{R}(r)\Bh$, respectively, then integrating by parts over  the  slab $\Omega=\mathbb{R}^{2}\times (0,1)$ and adding them up, 
one obtains	
	\begin{equation}\label{muinteeq}
		\begin{split}
		&\int_{\Omega}(|\nabla\Bu|^{2}+|\nabla\Bh|^{2})\varphi_{R} \, d\Bx\\
		=&-\int_{\OR}\nabla\varphi_{R}
		\cdot \nabla\Bu\cdot \Bu \,d\Bx-\int_{\OR}\nabla\varphi_{R}
		\cdot \nabla\Bh\cdot \Bh \,d\Bx + \frac{1}{2}\int_{\OR}(|\Bu|^{2}+|\Bh|^{2})\Bu\cdot \nabla\varphi_{R} \,d\Bx \\
	   &+ \int_{\OR} (P-P_{R}) \Bu \cdot \nabla \varphi_R \,d\Bx-\int_{\OR}
	  (\Bu\cdot \Bh) (\Bh\cdot \nabla\varphi_{R}) \,d\Bx\\
	  =&:J_{1}+J_{2}+J_{3}+J_{4}+J_{5},
		\end{split}
	\end{equation}
where $P_{R}=\dfrac{1}{|\OR|}\displaystyle\int_{\OR}P\,d\Bx$.
The reason we choose $P-P_{R}$  instead of $P$ is that $P-P_{R}$ is a $z$-variable periodic function in $\OR$.  Since $(\Bu,\Bh)$  is a $z$-variable periodic  function in  $\mathbb{R}^{3}$,  the momentum equation $\eqref{eqsteadyMHD}_{1}$ yields $\nabla P(r,\theta,z+1)-\nabla P(r,\theta,z)=0$, which leads to $ P(r,\theta,z+1)-P(r,\theta,z)=c$, where $c$ is a constant. So one has 
\[
\begin{split}
P(r,\theta, z+1)-\frac{1}{|\OR|}\int_{1}^{2}\int_{B_{R}\backslash B_{R-1}}&P(r,\theta,z+1)\, dSd(z+1)\\
=&P(r,\theta, z)-\frac{1}{|\OR|}\int_{0}^{1}\int_{B_{R}\backslash B_{R-1}}P(r,\theta,z)\, dSdz,
\end{split}
\]
i.e. $P-P_{R}$ is a periodic function with respect to $z$ variable, the boundary terms will vanish when we integrate by parts.

The straightforward computations give
	\begin{equation}\label{eqpressequr}
	J_{4}
		= -\int_0^1 \int_0^{2\pi} \int_{R-1}^R (P-P_{R})u^{r}r \, dr d\theta dz.
	\end{equation}
	Using the divergence free condition \eqref{eqdivecylinc}  and helically symmetric  properties \eqref{dhelicdef}, for all  $0\leq r<+\infty$,  yields
	\begin{equation}\label{eqA116}
		\begin{split}
			\partial_{r}\int_{0}^{1}ru^{r}\,dz=-\int_{0}^{1}\partial_{\theta}u^{\theta}+\partial_{z}(ru^{z})\,dz=-\int_{0}^{1}\dfrac{1}{2\pi}\partial_{z}u^{\theta}+\partial_{z}(ru^{z})\,dz =0.
		\end{split}
	\end{equation}
	This implies
	\begin{equation}\label{eqA117}
		\int_{0}^{1}ru^{r}\, dz=0
		\quad \text{and}\quad \int_{0}^{1} \int_{R-1}^{R}ru^{r}\, drdz=0.
	\end{equation}
Similarly, it holds that
	\begin{equation}\label{eqA11emagndiv}
	\int_{0}^{1}rh^{r}\, dz=0
	\quad \text{and}\quad \int_{0}^{1} \int_{R-1}^{R}rh^{r}\, drdz=0.
\end{equation}
	It follows from \eqref{eqA117} and Lemma \ref{Bogovskii} that for every fixed $\theta\in [0,2\pi]$,
there exists a vector valued function $\Ps_{R,\theta}(r,z)\in H^{1}_{0}(D_{R}; \mathbb{R}^{2})$ satisfying
 	\begin{equation}\label{eqA118}
 	\partial_{r}\Psi_{R,\theta}^{r}
 	+\partial_{z}\Psi_{R,\theta}^{z}
 	=ru^{r},
 \end{equation}
 together with the estimate
 \begin{equation}\label{eqA119}
 	\|(\partial_{r}, \partial_z) \Ps_{R,\theta}\|_{L^{2}(D_{R})} \leq C\|ru^{r}\|_{L^{2}(D_{R})},
 \end{equation}
 where $C$ is independent of $\theta$ and $R$.
 Owing to \eqref{eqA117}, the Poincar\'{e} inequality for $u^{r}$
 \begin{equation}\label{eqA122}
 	\|u^{r}\|_{L^{2}(\OR)}\leq	C\|\partial_{z}u^{r}\|_{L^{2}(\OR)} \end{equation}
 holds.
 This, together with \eqref{eqA119}, gives
 \begin{equation}\label{eqA123}
 	\|(\partial_{r}, \partial_z)\Ps_{R,\theta}\|_{L^{2}(\DR)} \leq C\|ru^{r}\|_{L^{2}(\DR)}\leq CR^{\frac{1}{2}}\|u^{r}\|_{L^{2}(\OR)}\leq CR^{\frac{1}{2}}\|\nabla\Bu\|_{L^{2}(\OR)}.
 \end{equation}
 Note that the Bogovskii map is a linear map  \cite[Section III.3]{GAGP11}. Hence there is a universal constant $C>0$ such that
 \begin{equation}\label{eqA124}
 	\|(\partial_{\theta}\partial_{r}, \partial_{\theta}\partial_{z}) \Ps_{R,\theta}\|_{L^{2}(\DR)} \leq C\|r\partial_{\theta}
 	u^{r}\|_{L^{2}(\DR)}\leq CR^{\frac{1}{2}}\|\partial_{z} u^{r}\|_{L^{2}(\OR)},
 \end{equation}
 where  the last inequality is due to \eqref{dhelicdef}.
Due to \eqref{eqA11emagndiv},  the Poincar\'{e} inequality for $h^{r}$
\begin{equation}\label{eqmagnepoin}
	\|h^{r}\|_{L^{2}(\OR)}\leq	C\|\partial_{z}h^{r}\|_{L^{2}(\OR)} \end{equation}
also holds.
 Furthermore, it follows from the periodicity of $P-P_{R}$, \eqref{eqpressequr} and \eqref{eqA118} that one obtains
 \begin{equation}\label{preeqes}
 	\begin{split}
 		J_{4}
 	    =&-\int_{0}^{1}
 		\int_{0}^{2\pi}
 		\int_{R-1}^{R}(P-P_{R})(\partial_{r}\Psi_{R,\theta}^{r}+\partial_{z}\Psi_{R,\theta}^{z}) \, drd\theta dz\\
 		=& \int_{0}^{1}\int_{0}^{2\pi}
 		\int_{R-1}^{R}(\partial_{r}P\Psi_{R,\theta}^{r}+ \partial_{z}P\Psi_{R,\theta}^{z})\, drd\theta dz=:J_{41}+J_{42}.
 	\end{split}
 \end{equation}
By virtue of  $\eqref{eqA8}_{1}$, $\eqref{eqA8}_{3}$ and integration by parts, one gets
 \begin{equation}\label{eqArP}
 	\begin{split}
 		& \int_{0}^{1}\int_{0}^{2\pi}\int_{R-1}^{R}\partial_{r}P\Psi^{r}_{R,\theta}\, drd\theta dz\\
 		=& -\int_{0}^{1}\int_{0}^{2\pi}\int_{R-1}^{R}\left(\partial_{r}u^{r}\partial_{r}\Psi^{r}_{R,\theta}+\partial_{z}u^{r}\partial_{z}\Psi^{r}_{R,\theta}
 		\right)\, drd\theta dz  -\int_{0}^{1}\int_{0}^{2\pi}\int_{R-1}^{R}\dfrac{1}{r^2}\partial_{\theta}u^{r}\partial_{\theta}\Psi^{r}_{R,\theta}\, drd\theta dz \\
 		&+\int_{0}^{1}\int_{0}^{2\pi}\int_{R-1}^{R}\left[\left(\dfrac{1}{r}\partial_{r}-\dfrac{1}{r^2}\right)u^{r}-\dfrac{2}{r^2}\partial_{\theta}u^{\theta}\right]\Psi^{r}_{R,\theta}\, drd\theta dz  \\
 		&-\int_{0}^{1}\int_{0}^{2\pi}\int_{R-1}^{R}\left[\left(u^{r}\partial_{r}+\dfrac{u^{\theta}}{r}\partial_{\theta}+u^{z}\partial_{z}\right)u^{r}-\dfrac{(u^{\theta})^2}{r}\right]\Psi_{R,\theta}^{r}\, drd\theta dz\\
 	  &+\int_{0}^{1}\int_{0}^{2\pi}\int_{R-1}^{R}\left[\left(h^{r}\partial_{r}+\dfrac{h^{\theta}}{r}\partial_{\theta}+h^{z}\partial_{z}\right)h^{r}-\dfrac{(h^{\theta})^2}{r}\right]\Psi_{R,\theta}^{r}\, drd\theta dz \\
 	  =&: J_{411}+J_{412}+J_{413}+J_{414}+J_{415}
 	\end{split}
 \end{equation}
 and
 \begin{equation}\label{eqAzP}
 	\begin{split}
 		&\int_{0}^{1}\int_{0}^{2\pi}\int_{R-1}^{R}\partial_{z}P\Psi^{z}_{R,\theta}\, drd\theta dz\\
 		=&	-\int_{0}^{1}\int_{0}^{2\pi}\int_{R-1}^{R}\left(\partial_{r}u^{z}\partial_{r}\Psi_{R,\theta}^{z}+\partial_{z}u^{z}\partial_{z}\Psi_{R,\theta}^{z}+\dfrac{1}{r^2}\partial_{\theta}u^{z}\partial_{\theta}\Psi_{R,\theta}^{z}\right)\, drd{\theta}dz\\
 		& -\int_{0}^{1}\int_{0}^{2\pi}\int_{R-1}^{R}\left[\left(u^{r}\partial_{r}+
 		\dfrac{u^{\theta}}{r}\partial_{\theta}+u^{z}\partial_{z}-\dfrac{1}{r}\partial_{r}\right)u^{z}\right]\Psi_{R,\theta}^{z}\, drd\theta dz\\
 		 &+\int_{0}^{1}\int_{0}^{2\pi}\int_{R-1}^{R}\left[\left(h^{r}\partial_{r}+\dfrac{h^{\theta}}{r}\partial_{\theta}+h^{z}\partial_{z}\right)h^{z}\right]\Psi_{R,\theta}^{z}\, drd\theta dz.
 	\end{split}
 \end{equation}
\emph{Step 2.} \emph {Saint-Venant type estimate.} Using helical identities \eqref{dhelicdef}, Poincar\'e inequality \eqref{eqA122} and \eqref{eqmagnepoin},   the estimates   \eqref{eqA123}--\eqref{eqA124}  and $\|u^{r}\|_{L^{2}(\OR)}\leq CR^{\frac{1}{2}}\|\Bu\|_{L^{\infty}(\OR)}$, one has
\begin{equation}\label{eqA126}
	\begin{split}
	|J_{411}|\leq	&C\|\nabla\Bu\|_{L^{2}(\DR)}\cdot   \|(\partial_{r}, \partial_z)\Ps_{R,\theta}\|_{L^{2}(\DR)}\\
		\leq&      CR^{-\frac{1}{2}}\|\nabla\Bu\|_{L^{2}(\OR)}\cdot   R^{\frac{1}{2}}\|u^{r}\|_{L^{2}(\OR)}\leq  CR^{\frac{1}{2}} \|\nabla\Bu\|_{L^{2}(\OR)}
	\end{split}
\end{equation}
and
\begin{equation}\label{eqA127}
	\begin{split}
	|J_{412}|
		\leq&CR^{-2}\|\partial_{\theta}u^{r}\|_{L^{2}(\DR)}\cdot   \|\partial_{\theta} \partial_z\Ps_{R,\theta}\|_{L^{2}(\DR)}\\ \leq&CR^{-\frac{5}{2}}\|\partial_{z}u^{r}\|_{L^{2}(\OR)}\cdot R^{\frac{1}{2}}\|\nabla\Bu\|_{L^{2}(\OR)}\\
		\leq& CR^{-2}\|\nabla\Bu\|^{2}_{L^{2}(\OR)}\leq C\|\nabla\Bu\|_{L^{2}(\OR)},
	\end{split}
\end{equation}	
where  the last inequality is obtained since the  assumption \eqref{21Dintass}. Furthermore, one has	
\begin{equation}\label{eqA128}
	\begin{split}
	|J_{413}|&\leq CR^{-1}(\| \nabla\Bu\|_{L^{2}(\DR)}+R^{-1}\| \partial_{z}u^{\theta}\|_{L^{2}(\DR)})
 \|\partial_z\Ps_{R,\theta}\|_{L^{2}(\DR)}\\
 &\leq	CR^{-1}(\| \nabla\Bu\|_{L^{2}(\DR)}+R^{-1}\| \nabla\Bu\|_{L^{2}(\DR)})
	\cdot R^{\frac{1}{2}} \|u^{r}\|_{L^{2}(\OR)}\\ 
	&\leq CR^{-1}R^{-\frac{1}{2}}\| \nabla\Bu\|_{L^{2}(\OR)}
		\cdot R^{\frac{1}{2}} \|u^{r}\|_{L^{2}(\OR)}
		\leq CR^{-\frac{1}{2}}\| \nabla\Bu\|_{L^{2}(\OR)},
	\end{split}
\end{equation}
where the first line we have used helical identities \eqref{dhelicdef}.
A similar argument shows that
\begin{equation}\label{eqA129}
	\begin{split}			   
	|J_{414}|	\leq&
	C\|\Bu\|_{L^{\infty}(\OR)}
	\left(R^{-\frac{1}{2}}\|\nabla\Bu\|_{L^{2}(\OR)}+R^{-\frac{3}{2}}\|u^{\theta}\|_{L^{2}(\OR)}\right)\cdot R^{\frac{1}{2}}\|u^{r}\|_{L^{2}(\OR)}\\
	\leq&
	CR^{\frac{1}{2}}\| \nabla\Bu\|_{L^{2}(\OR)}
	\end{split}
\end{equation}
and
\begin{equation}\label{eqnewmg}
	\begin{split}			   
		|J_{415}|	\leq&
		C\|\Bh\|_{L^{\infty}(\OR)}
		\left(R^{-\frac{1}{2}}\|\nabla\Bh\|_{L^{2}(\OR)}+R^{-\frac{3}{2}}\|h^{\theta}\|_{L^{2}(\OR)}\right)\cdot R^{\frac{1}{2}}\|u^{r}\|_{L^{2}(\OR)}\\
		\leq&
		CR^{\frac{1}{2}}(\| \nabla\Bh\|_{L^{2}(\OR)}+\| \nabla\Bu\|_{L^{2}(\OR)}).
	\end{split}
\end{equation}
Collecting the estimates \eqref{eqA126}--\eqref{eqnewmg} gives
\begin{equation}\label{eqA130}
	\left|J_{41}\right|
	\leq 	CR^{\frac{1}{2}}(\| \nabla\Bh\|_{L^{2}(\OR)}+\| \nabla\Bu\|_{L^{2}(\OR)}).
\end{equation}
	Similarly, it holds that
\begin{equation}\label{eqA131}
	\left|J_{42}\right|
	\leq 	CR^{\frac{1}{2}}(\| \nabla\Bh\|_{L^{2}(\OR)}+\| \nabla\Bu\|_{L^{2}(\OR)}).
\end{equation}

Next we  estimate the other terms on the right hand side of \eqref{muinteeq}.
Using  H{\"o}lder inequality  yields
\begin{equation}\label{basces1}
	\begin{split}
	\left|J_{1} \right|&\leq
	C\|\nabla\Bu\|_{L^{2}(\OR)}\cdot \|\Bu\|_{L^{2}(\OR)}\\
	 &\leq C\|\nabla\Bu\|_{L^{2}(\OR)}\cdot R^{\frac{1}{2}}\|\Bu\|_{L^{\infty}(\OR)}\leq CR^{\frac{1}{2}}\|\nabla\Bu\|_{L^{2}(\OR)}.
		\end{split}
\end{equation}
Likewise,
\begin{equation}
|J_{2}|\leq CR^{\frac{1}{2}}\|\nabla\Bh\|_{L^{2}(\OR)}.
\end{equation}
It follows from	Poincar\'{e} inequality \eqref{eqA122} that one obtains
\begin{equation}\label{basces2}
	\left|J_{3}\right|\leq C(\|\Bu\|^{2}_{L^{\infty}(\OR)}+\|\Bh\|^{2}_{L^{\infty}(\OR)}) \|u^{r}\|_{L^{2}(\OR)}\cdot R^{\frac{1}{2}}\leq CR^{\frac{1}{2}}\|\nabla\Bu\|_{L^{2}(\OR)}.
\end{equation}
Analogously, using the Poincar\'{e} inequality \eqref{eqmagnepoin} yields
\begin{equation}\label{estieqj5}
|J_{5}|\leq C\|\Bu\|_{L^{\infty}(\OR)}\|\Bh\|_{L^{\infty}(\OR)}\|h^{r}\|_{L^{2}(\OR)}\cdot R^{\frac{1}{2}}\leq CR^{\frac{1}{2}}\|\nabla\Bh\|_{L^{2}(\OR)}.
\end{equation}
Combining  \eqref{eqA130}--\eqref{estieqj5}, one arrives at
\begin{equation}\label{eqbaesin}
	\int_{\Omega}(|\nabla\Bu|^{2}+|\nabla\Bh|^{2})\varphi_{R} \, d\Bx
	\leq CR^{\frac{1}{2}}(\|\nabla\Bu\|_{L^{2}(\OR)}+\|\nabla\Bh\|_{L^{2}(\OR)}).
\end{equation}

Let
\begin{equation}\label{gDinsgro}
	Y(R)=
	\int_{0}^{1}\iint_{\mathbb{R}^{2}}
	(|\nabla\Bu|^{2}+|\nabla\Bh|^{2})\varphi_{R}\left(\sqrt{x_{1}^{2}+x_{2}^{2}}\right)
	\, dx_{1} dx_{2} dx_{3}.
\end{equation}
Straightforward computations give
\begin{equation}\label{eqA50}
	Y(R)=\int_{0}^{1}\int_{0}^{2\pi}
	\left(\int_{0}^{R-1}(|\nabla\Bu|^{2}+|\nabla\Bh|^{2})r\,dr+\int_{R-1}^{R}(|\nabla\Bu|^{2}+|\nabla\Bh|^{2})(R-r)r\, dr\right)d\theta dz
\end{equation}
and
\begin{equation}\label{eqA51}
	Y^{\prime}(R)=\int_{\OR}(|\nabla\Bu|^{2}+|\nabla\Bh|^{2})\,d\Bx.
\end{equation}
Hence the estimate \eqref{eqbaesin} can be written
as
\begin{equation}\label{eqA52}
	Y(R)\leq CR^{\frac{1}{2}} \left[Y^{\prime}(R)\right]^{\frac{1}{2}}.
\end{equation}
If $\nabla\Bu$ and $\nabla\Bh$ are not identically equal to zero, then $Y(R)> 0$ for $R\geq R_{0}$ with some $R_{0}>0$, and one has
\begin{equation}\label{eqA53}
	\dfrac{1}{C^{2}R}\leq
	\left(-\dfrac{1}{Y(R)}\right)^{\prime}.
\end{equation}
Integrating it over $(R_{0},R)$ for large $R_{0}$, one arrives at
\begin{equation}\label{eqA54}
	\dfrac{1}{C^{2}}\ln\dfrac{R}{R_{0}}\leq
	-\dfrac{1}{Y(R)}+\dfrac{1}{Y(R_{0})}\leq \dfrac{1}{Y(R_{0})}.
\end{equation}
This leads to a contradiction when $R$ is sufficiently large. Hence $\nabla \Bu=\nabla \Bh\equiv0$, therefore  $\Bu$ and $\Bh$ are constant vectors.	Moreover, the helical properties \eqref{helidef} give
\begin{equation}
	\begin{split}
	\Bu=&u^{r}(r,\theta,z) \Be_{r}+
	  u^{\theta}(r,\theta,z)\Be_{\theta}+
	   u^{z}(r,\theta,z)\Be_{z}\\
	 =&v^{r}\left(r, \frac{1}{2\pi}\theta+z\right) \Be_{r}+
	 v^{\theta}\left(r, \frac{1}{2\pi}\theta+z\right)\Be_{\theta}+
	 v^{z}\left(r, \frac{1}{2\pi}\theta+z\right)\Be_{z}.
\end{split}
\end{equation}
Hence one has $u^{r}=u^{\theta}\equiv 0$ and $u^{z}\equiv C_{1}$. The similar argument to the magnetic field $\Bh$ leads to $h^{r}=h^{\theta}\equiv 0$ and $h^{z}\equiv C_{2}$.  This finishes the proof of Lemma \ref{Le:32}.
\end{proof}
Now we are ready to give a proof of  Theorem \ref{th:01}. According to the proof of Lemma \ref{Le:32}, we  need to remove the  finite Dirichlet integral assumption \eqref{21Dintass}.
\begin{proof}[Proof for  Theorem \ref{th:01}]
 Noting that the equality \eqref{muinteeq} still holds.
	The proof is almost the same as that in  Lemma \ref{Le:32}, except that
	\begin{equation}\label{eqA133}
		\begin{split}
		|J_{412}|=&\left| \int_{0}^{1}\int_{0}^{2\pi}
			\int_{R-1}^{R}\dfrac{1}{r^{2}}\partial_{\theta} u^{r}\partial_{\theta} \Psi_{R,\theta}^{r}\,drd\theta dz
			\right|\\
			\leq& CR^{-\frac{5}{2}}\|\partial_{z}u^{r}\|_{L^{2}(\OR)}\cdot R^{\frac{1}{2}}\|\nabla\Bu\|_{L^{2}(\OR)}
			\leq C\|\nabla\Bu\|^{2}_{L^{2}(\OR)}.
		\end{split}
	\end{equation}	
From the proof of Lemma \ref{Le:32}, it holds that
	\begin{equation}\label{eqA134}
		Y(R)\leq  C_{1}Y^{\prime}(R)
		+C_{2}R^{\frac{1}{2}}
		\left[Y^{\prime}(R)\right]^{\frac{1}{2}},
	\end{equation}
	where $Y(R)$ is defined in \eqref{gDinsgro}. Hence one has
	\begin{equation}\label{eqA135}
		\left[Y^{\prime}(R)\right]^{\frac{1}{2}}
		\geq \dfrac{-C_{2}R^{\frac{1}{2}}+\sqrt{C_{2}^{2}R+4C_{1}Y(R)}}{2C_{1}}
		\geq
		\dfrac{Y(R)}{\sqrt{C_{2}^{2}R+ 4C_{1}Y(R)}}.
	\end{equation}
	If $\nabla\Bu$ and $\nabla\Bh$ are not identically  zero,  $Y(R)>0$ for $R$ large enough. It follows from \eqref{eqA135} that one obtains
	\begin{equation}\label{eqA136}
		\left[C^{2}_{2}RY^{-2}(R)+4C_{1}
		Y^{-1}(R)\right]Y^{\prime}(R)\geq 1.
	\end{equation}

	Let $M$ be a large number satisfying $M^{-1}C^{2}_{2}\leq \frac{1}{4}$. According to Lemma \ref{Le:32}, there exists a constant $R_{0}>2$ such that $Y(R_{0})\geq M$, otherwise $\nabla\Bu=\nabla\Bh\equiv0$. For every $R>R_{0}$, integrating \eqref{eqA136} over $\left[R, 2R\right]$ gives
	\begin{equation}\label{eqA137}
		2R\cdot C^{2}_{2}\left[\dfrac{1}{Y(R)}
		-\dfrac{1}{Y(2R)}\right]+4C_{1}\ln \dfrac{Y(2R)}{Y(R)}\geq R.
	\end{equation}
	Since $Y(R)\geq M$, it holds that
	\begin{equation}\label{eqA138}
		\dfrac{Y(2R)}{Y(R)}\geq \text{exp}\left\{ \dfrac{R}{8C_{1}}\right\}.
	\end{equation}
	This implies the exponential growth of $\|\nabla\Bu\|_{L^{2}(\Omega_{R})}$ or $\|\nabla\Bh\|_{L^{2}(\Omega_{R})}$. However,  Lemma \ref{le:pebu} shows that $\nabla \Bu$ and $\nabla \Bh$ are bounded in
	$\Omega$. So  $\|\nabla\Bu\|_{L^{2}(\Omega_{R})}$ and $\|\nabla\Bh\|_{L^{2}(\Omega_{R})}$ are at most linearly growing in $R$. This contradiction implies $\nabla\Bu=\nabla\Bh\equiv0$. It follows from the proof of  Lemma \ref{Le:32} that the velocity field $\Bu$ and the magnetic field $\Bh$  must be  constant vectors of the form $(0, 0, C_{1})$ and $(0, 0, C_{2})$, respectively.  Hence the proof of Theorem \ref{th:01} is completed.
   \end{proof}
\begin{appendices}
	\section*{Appendix: Notes on Remark \ref{remarkreq1}}\label{proofRE1}
	
	Note that  the assumption on magnetic field $\Bh$  improved from the helical symmetry   to periodic in $z$-variable, the proof of Remark \ref{remarkreq1} has two differences  as in the   Theorem \ref{th:01}.  Specifically, (i) the Poincar\'{e} inequality \eqref{eqmagnepoin} for $h^{r}$.	Using the divergence free condition \eqref{eqdivecylinc} yields     
	\begin{equation}\label{eqA52appe}
		\begin{split}
			\partial_{r}\int_{0}^{1}\int_{0}^{2\pi}ru^{r}\,d\theta dz=-\int_{0}^{1}\partial_{\theta}u^{\theta}+\partial_{z}(ru^{z})\,dz=0,\ \ \text{for all}\ \ 0\leq r<+\infty.
		\end{split}
	\end{equation}
	This implies
 \[
\int_{0}^{1}\int_{0}^{2\pi}ru^{r}\,d\theta dz=0
	\quad \text{and}\quad \int_{0}^{1} \int_{0}^{2\pi}\int_{R-1}^{R}ru^{r}\, drd \theta dz=0.
	\]
Hence, it holds the Poincar\'{e} inequality \eqref{eqmagnepoin} for $h^{r}$
\[
	\|h^{r}\|_{L^{2}(\OR)}\leq	C\|\partial_{z}h^{r}\|_{L^{2}(\OR)}.
\]
(ii) The estimate $|J_{415}|$.  The matrix $\nabla\Bu$ in the cylindrical coordinates can be written as a form of tensor product as follows,
\[
\begin{split}
	\nabla\Bu
	= &
	\,\,\, \partial_{r}u^{r}\Be_{r}\otimes\Be_{r}+\left(\frac{1}{r}\partial_{\theta}u^{r}-\frac{u^{\theta}}{r}\right)\Be_{r}\otimes\Be_{\theta}+\partial_{z}u^{r}\Be_{r}\otimes\Be_{z}  \\
	&+\partial_{r}u^{\theta}\Be_{\theta}\otimes\Be_{r}+\left(\frac{1}{r}\partial_{\theta}u^{\theta}+\frac{u^{r}}{r}\right)\Be_{\theta}\otimes\Be_{\theta}+\partial_{z}u^{\theta}\Be_{\theta}\otimes\Be_{z} \\
	&+\partial_{r}u^{z}\Be_{z}\otimes\Be_{r}+\frac{1}{r}\partial_{\theta}u^{z}\Be_{z}\otimes\Be_{\theta}+\partial_{z}u^{z}\Be_{z}\otimes\Be_{z}.
\end{split}
\]
So one obtains
\begin{equation}\label{divfes1}
	|\partial_{r}u^{r}|+|\partial_{r}u^{\theta}|+|\partial_{r}u^{z}|
	+|r^{-1}\partial_{\theta}u^{z}|+|\partial_{z}u^{r}|+|\partial_{z}u^{\theta}|+|\partial_{z}u^{z}|\leq C|\nabla\Bu|
\end{equation}
and
\begin{equation}\label{divfes2}
	|r^{-1}\partial_{\theta}u^{r}-r^{-1}u^{\theta}|+|r^{-1}\partial_{\theta}u^{\theta}+r^{-1}u^{r}|\leq C|\nabla \Bu|.
\end{equation}
Hence, it holds that
\begin{equation}
	\begin{split}
|J_{415}|= &\left|\int_{0}^{1}\int_{0}^{2\pi}\int_{R-1}^{R}\left[\left(h^{r}\partial_{r}+\dfrac{h^{\theta}}{r}\partial_{\theta}+h^{z}\partial_{z}\right)h^{r}-\dfrac{(h^{\theta})^2}{r}\right]\Psi_{R,\theta}^{r}\, drd\theta dz\right|\\
 \leq&C\|\Bh\|_{L^{\infty}(\DR)}
 \|(\partial_{r},\partial_{z})h^{r}+r^{-1}(\partial_{\theta}h^{r}-h^{\theta})\|_{L^{2}(\DR)}\cdot \|\partial_{z}\Psi_{R,\theta}^{r}\|_{L^{2}(\DR)}\\
 \leq&C\|\Bh\|_{L^{\infty}(\OR)}\cdot
R^{-\frac{1}{2}}\|\nabla\Bh\|_{L^{2}(\OR)}\cdot R^{\frac{1}{2}}\|u^{r}\|_{L^{2}(\OR)}
\leq C
R^{\frac{1}{2}}\|\nabla\Bh\|_{L^{2}(\OR)}.
		\end{split}
\end{equation}
\end{appendices}	

	{\bf Acknowledgement.}
 The author would like to thank  Professor Chunjing Xie and Professor Yun Wang for helpful discussion    and constant encouragement.

\end{document}